\documentclass[11pt,a4paper]{amsart}
\usepackage[utf8]{inputenc}
\usepackage{amsmath}
\usepackage{amssymb}
\usepackage[a4paper,
bindingoffset=0.2in,
left=2.4cm,
right=2.4cm,
top=3.5cm,
bottom=3.5cm,
footskip=.25in]{geometry}
\usepackage{amsthm}
\usepackage{tikz-cd}
\usepackage{cleveref}

\DeclareMathOperator{\im}{\operatorname{im}}

\let\OLDthebibliography\thebibliography  
\renewcommand\thebibliography[1]{
	\OLDthebibliography{#1}
	\setlength{\parskip}{0pt}
	\setlength{\itemsep}{0pt plus 0.3ex}
}

\theoremstyle{definition}
\newtheorem{thm}{Theorem}[section]
\newtheorem{prop}[thm]{Proposition}

\newtheorem{defi}[thm]{Definition}
\newtheorem{rem}[thm]{Remark}
\newtheorem{ex}[thm]{Example}

\newtheorem{question}[thm]{Question}

\newcommand{\del}{\partial}
\newcommand{\delbar}{\overline{\partial}}
\newcommand{\deldelbar}{\partial \overline{\partial}}

\newcommand{\spn}{\mathrm{span}}
\newcommand{\cH}{\mathcal{H}}
\newcommand{\pr}{\operatorname{pr}}
\newcommand{\C}{\mathbb{C}}

\newcommand{\CC}{\mathbb{C}}
\newcommand{\su}{\mathfrak{su}}
\newcommand{\Z}{\mathbb{Z}}

\newcommand{\cM}{\mathcal{M}}

\title[Bigraded notions of formality]{Bigraded notions of formality and Aeppli--Bott--Chern--Massey products}

\begin{document}
	\author{A. Milivojevi\'c}
	\address{University of Waterloo, Faculty of Mathematics}
	\email{amilivojevic@uwaterloo.ca}
	\author{J. Stelzig}
	\address{ Mathematisches Institut der Ludwig-Maximilians-Universit\"at M\"unchen}
	\email{jonas.stelzig@math.lmu.de}
	
	
	\begin{abstract}
		We introduce and study notions of bigraded formality for the algebra of forms on a complex manifold, along with their relation to higher Aeppli--Bott--Chern--Massey products which extend the case of triple products studied by Angella--Tomassini. We show that these Aeppli--Bott--Chern--Massey products on complex manifolds pull back non-trivially to the blow-up along a complex submanifold, as long as their degree is less than the real codimension of the submanifold. 
	\end{abstract}
	
	\maketitle
	
	\section{Introduction}
	From an early stage in the development of rational homotopy theory, there have been fruitful interactions with complex geometry: Deligne--Griffiths--Morgan--Sullivan famously proved that all compact K\"ahler manifolds are rationally formal \cite{DGMS75}, \cite{Su77}. Neisendorfer--Taylor adapted the notions of models and formality to the holomorphic category \cite{NT78} while giving some preference to the antiholomorphic differential $\delbar$ as necessitated by analogy to a singly-graded theory. Naturally, the question arises whether one can build a more symmetric theory which treats both differentials on equal footing.
	
	The first hints that such a theory may be possible and meaningful arose in the works of Angella--Tomassini \cite{AT15}, where an ad-hoc definition of new, symmetric Massey like triple products was given. These use genuinely new cohomological invariants of bicomplexes not defined for simple complexes, namely the input classes live in Bott-Chern cohomology and the output is an Aeppli-cohomology class. This left open two natural questions: 
	
	\begin{enumerate}
	    \item How can one generalize the definition of the triple products to an arbitrary number of inputs?
	    \item What is the homotopy--theoretic notion of formality that is obstructed by these products?
	\end{enumerate}
	In this article we answer the above questions, thereby giving new holomorphic invariants of complex manifolds with a homotopy--theoretic flavor.
	
	First, we identify the Aeppli--Bott--Chern--Massey (ABC--Massey) triple products as one member of a sequence of higher ABC--Massey products, employing a chain complex studied by Schweitzer and guided by a general framework for defining Massey--like products developed by Massey himself in the 1950's. In this framework, the ABC--Massey triple products studied in \cite{AT15} can be thought of as associated to a 1--simplex, while the quadruple, quintuple, etc. products are associated to higher-dimensional simplices. 
	
	We then identify two homotopy--theoretic notions of formality for bigraded bidifferential algebras, one stronger than the other, but equivalent under the $\del\delbar$-lemma (e.g. on compact K\"ahler manifolds). They are obstructed by the presence of ABC--Massey products. The relations between these, a metric--dependent formality condition akin to geometric formality, and ordinary (de Rham) formality are discussed through examples. Interestingly, it is not clear whether compact K\"ahler manifolds are formal in this new sense, and there are examples of non-formal $\del\delbar$-manifolds.
	
	Finally, we study the pullback behavior of the new higher operations. Babenko--Taimanov \cite{BT00} showed that Massey products on symplectic manifolds are preserved under symplectic blow-up along a submanifold, as long as their degree is less than twice the real codimension of the submanifold. We extend this to complex blow-ups and ABC--Massey products.  Adapting an argument of Taylor, we further observe that ABC--Massey triple products always pull back non-trivially under nonzero degree holomorphic maps of compact complex manifolds.

	\subsection*{Acknowledgements} This research was supported through the ``Research in Pairs'' program at the Mathematisches Forschungsinstitut Oberwolfach in December 2021. We thank the MFO for the excellent working conditions provided there. The first named author would likewise thank the Max Planck Institute for Mathematics. We are grateful to the Institut de Matemàtica, Universitat de Barcelona, and to the City University of New York Graduate Center for their generous hospitality; we thank Joana Cirici and Scott Wilson for many useful discussions and comments. We further thank Michael Albanese, John Morgan, Dennis Sullivan, and Leopold Zoller for stimulating conversations, and the anonymous referees for their helpful comments.
	
	\subsection*{Related work} 
The thesis of Nicoletta Tardini \cite{Ta17} contains a different notion of higher length Massey products that takes an odd number of input Bott--Chern classes and produces an Aeppli class.

	A more systematic treatment of the notions of (minimal) models and a model category structure for bigraded, bidifferential algebras that naturally gives rise to the notions of formality considered here is performed by the second named author in \cite{Ste23}.
	\section{Preliminaries}
	
	One of the basic objects we consider are double complexes, i.e. bigraded complex vector spaces with differentials, suggestively denoted $\del$ and $\delbar$, of bidegree $(1,0)$ and $(0,1)$ respectively, such that $(\del + \delbar)^2 = 0$.  We recall that double complexes admit direct sum decompositions into well-understood indecomposable subcomplexes (``squares and zigzags''), see \cite{KQ20}, \cite{Ste21}.

    Recall that the Bott--Chern and Aeppli cohomology are functors from the category of bicomplexes to the category of bigraded complex vector spaces given by

    $$H_{BC} = \frac{\ker \del \cap \ker \delbar}{\im \deldelbar}, \ \ H_A = \frac{\ker \deldelbar}{\im \del + \im \delbar}.$$

    \begin{defi}\label{def: bigraded quiso}
        A map of double complexes $\varphi: A\longrightarrow B$ is called a bigraded quasi--isomorphism if it induces an isomorphism in Bott--Chern and Aeppli cohomology.
    \end{defi}

    \begin{rem} A map induced by a holomorphic map of \emph{compact} complex manifolds that induces an isomorphism on Bott--Chern cohomology automatically induces an isomorphism on Aeppli cohomology (and vice versa) by Serre duality \cite{S07}, and is hence a bigraded quasi--isomorphism.
    \end{rem}
    
	\begin{defi}\label{def: coh. fun.}
		A cohomological functor is a linear functor from the category of double complexes to the category of vector spaces which sends direct sums of squares to the zero vector space.
	\end{defi}
	
	\begin{prop}\label{prop: E1-isos}
		Let $f:A\to B$ be an bigraded quasi--isomorphism of bounded double complexes.
		\begin{enumerate}
			\item The induced map $f^{\otimes n}:A^{\otimes n}\to B^{\otimes n}$ is an bigraded quasi--isomorphism.
			\item For any cohomological functor $H$, the induced map $H(f)$ is an isomorphism.
		\end{enumerate}
	\end{prop}
	\begin{proof}
		The first statement is a combination of \cite[Theorem 1.21]{Ste23} and \cite[4.2.15]{Hov99}, c.f. \cite[Eq. 1.5.]{Ste23}. The second is \cite[Corollary 1.25]{Ste23}.
	\end{proof}

    \begin{rem}\label{rem: E1 vs ABC quiso}
    We recall \cite{Ste21} that a map of double complexes $\varphi: A\longrightarrow B$ is called an $E_1$--isomorphism if it induces an isomorphism in row and column cohomology. As a direct consequence of \Cref{prop: E1-isos}, a bigraded quasi--isomorphism is an $E_1$--isomorphism. For bounded double complexes, the converse is also true, \cite[Theorem 1.21]{Ste23}.	
	\end{rem}

	\begin{rem}\label{realstructure}
		In general, the notions of bigraded quasi--isomorphism and $E_1$--isomorphism impose stronger conditions than that of inducing an isomorphism in column (``Dolbeault'') cohomology only (see the notion of \emph{$E_0$-quasi-isomorphism} in \cite{CSLW20}). However, if both double complexes are equipped with a real structure, i.e. an antilinear involution $\sigma$ such that $\sigma A^{p,q}=A^{q,p}$ and $\sigma\del\sigma=\delbar$, and we consider maps compatible with this structure ($\sigma\varphi=\varphi\sigma$), then the condition of being an $E_1$--isomorphism is the same as inducing an isomorphism in column cohomology only. For bounded complexes, it is therefore also the same as that of bigraded quasi--isomorphism. In particular, this applies to the case of $A=A_X$, $B=A_Y$ being the double complexes of forms on complex manifolds $X,Y$, and $\varphi=f^*$ induced by a holomorphic map $Y\xrightarrow{f} X$, or to the inclusion of forms $A^G_X\subseteq A_X$ invariant under a group acting by biholomorphisms.
	\end{rem}

	The forms on a complex manifold have the additional structure of a graded--commutative product, where the graded--commutativity takes into consideration only the total degree, together with the differentials $\del$ and $\delbar$ being (graded, again with respect to the total grading) derivations. We refer to such a structure as a \emph{(graded--)commutative bigraded bidifferential algebra}, or \textbf{cbba} for short. A map of cbba's whose underlying map of double complexes is an bigraded quasi--isomorphism will be called a \textbf{weak equivalence} (this model--categorical terminology is justified in \cite{Ste23}).
	
	There are augmented versions of all the above objects; both will be relevant in the following sections.
	
	\section{Aeppli--Bott--Chern higher products}\label{ABCsection}
	
	Recall the Aeppli--Bott--Chern--Massey (ABC--Massey) triple product, as defined in \cite[Definition 2.1]{AT15}. For Bott--Chern cohomology classes $a, b, c \in H_{BC}(X)$ such that $ab = bc = 0$, take representatives $\alpha, \beta, \gamma$ for $a,b,c$ respectively, and choose forms $x,y$ such that $\deldelbar x = \alpha \beta$ and $\deldelbar y = \beta \gamma$. Then the triple product is the coset in $H_A(X)/\left( aH_A(X) + cH_A(X) \right)$ corresponding to the Aeppli cohomology class $[\alpha y - x \gamma] \in H_A(X)$. Note the different sign convention than in \cite{AT15}; the results of loc. cit. carry through verbatim for this different sign convention. 
	
	We adapt Massey's spectral sequence construction of triple and higher order Massey products to the setting of complex manifolds. This construction is basepoint-dependent, and hence will give invariants of pointed complex manifolds. However, the construction of the Massey product as a differential in a certain spectral sequence reviewed below,  informs one how to define the usual, basepoint-independent, ``ad hoc'' Massey products. In either case, we recover the ordinary and Dolbeault--Massey products \cite{CT15}, along with ABC--Massey triple products (see e.g. \cite{TT14}, \cite{AT15}). This allows us to define higher ABC--Massey products in the spectral sequence setup. We give an explicit definition and formula for ad hoc quadruple ABC--Massey products, which can in principle be extended to quintuple and higher products. Generally the spectral sequence Massey products have a larger indeterminacy but also a larger domain of definition than the ad hoc ones.
		
	\subsection{Eilenberg--Moore spectral sequence}
	We fix an augmented abstract simplicial complex $K$, i.e. a collection of subsets (including the empty set) of some fixed set such that for every $\sigma\in K$ and $\tau\subseteq \sigma$, we also have $\tau\in K$. We will later only consider the case that $K$ is the standard simplex, i.e. $K=\Delta^n=$ all subsets of $\{0,...,n\}$, but the construction works in this greater generality and gives, in principle, additional invariants.
	
	We further fix a linear functor $S$ from double complexes to simple complexes such that its composition with taking cohomology is a cohomological functor. On a first reading, the reader may want to assume that this associates to a double complex $A$ its total complex. This choice of $A$ will recover the Massey products as presented in \cite[\textsection \textsection 3--4]{M58}.
	
	Now let $A$ be an augmented cbba and write $A^+=\ker(A\to\C)$ for its augmentation ideal. This is a cbba without unit. The principal example we have in mind is $A=A_X$ for a complex manifold and the augmentation given by the restriction of forms to some basepoint $x\in X$.
	We may now form a new triple complex $C_{\ast}^{\ast,\ast}(A,K)$ as follows: For every fixed $p\in\Z$, define 
	\begin{equation}\label{simplicial bicomplex}
		C_p(A,K):=\bigoplus_{\sigma\in K, |\sigma|=p}\underbrace{A^+\otimes...\otimes A^+}_{p+2\text{ times}}
	\end{equation}
	(where we consider the empty set to have cardinality $-1$). Each $C_p(A,K)$ inherits the structure of a double complex from $A$, by using the usual induced differentials and gradings on tensor products and direct sums of double complexes. This gives the upper grading on $C_{\ast}^{\ast,\ast}(A,K)$. The differential $\delta$ in the direction of the lower grading is given as follows: Denote the summand belonging to a given $p$-simplex $\sigma$ by $C_p(A,K)_{\sigma}$. If $\sigma_i$ is the face obtained by omitting the $i$-th vertex, then on elementary tensors in $C_p(A,K)_\sigma$, one sets 
	\[
	\delta (a_0\otimes...\otimes a_p)=\sum_{\sigma_i\in K}(-1)^i a_0\otimes...\otimes a_i\cdot a_{i+1}\otimes...\otimes a_p
	\]
	and extends linearly. This defines a map of double complexes $C_p(A,K)\to C_{p-1}(A,K)$ and satisfies $\delta^2=0$. 
	
	Now we may apply $S$ to each $C_p^{\ast,\ast}(A,K)$ to obtain a simple complex $C_p^{\ast}(A,K,S)$ and, consequently, a double complex $C_{\ast}^{\ast}(A,K,S)$ (with commuting differentials instead of anticommuting ones). Note that the ``horizontal'' differential $\delta$ is of degree $-1$ and the vertical differential is of degree $1$. 
	
	\begin{defi}
		The \textbf{Eilenberg--Moore spectral sequence} $\{EM(A,K,S)^{\ast,\ast}_r, d_r\}_{r\geq 1}$ is the spectral sequence associated with the filtration of $C_\ast^{\ast}(A,K,S)$ by horizontal degree. Its differentials are of degree $|d_r|=(-r,-r+1)$ and the spaces on the first page are given by 
		\[
		EM(A,K,S)^{p,q}_1=H^q(C_p^\ast(X,K,S))=\bigoplus_{\sigma\in K,|\sigma|=p}H^q(S(A^+\otimes \cdots \otimes A^+)).
		\]
		with differential $d_1$ induced by $\delta$.
	\end{defi}
	
	\begin{rem}\label{functoriality}
	By construction, this spectral sequence is functorial for maps of augmented cbba's, and by (the augmented version of) \Cref{prop: E1-isos}, a weak equivalence of augmented cbba's gives rise to isomorphic spectral sequences. 
	\end{rem}
	
	\begin{ex}(\cite{M58})
		Let $S$ be the total-complex functor and $K=\Delta^n$ the standard simplex. Then for $n=0$, The first page looks as follows:
		\[
		H(A^+_{tot})\overset{d_1}{\longleftarrow}H(A^+_{tot})\otimes H(A^+_{tot}),
		\]
		where we suppress the vertical grading in the notation and the differential is given by multiplication. For $n=1$, we obtain as first page
		\[
		H(A^+_{tot})\overset{d_1}{\longleftarrow}H(A^+_{tot})^{\otimes 2}\oplus H(A^+_{tot})^{\otimes 2}\overset{d_1}{\longleftarrow}H(A^+_{tot})^{\otimes 3},
		\]
		where the first map is given by $(a\otimes b, c\otimes d)\mapsto ab+cd$ and the second one by $(a\otimes b\otimes c)\mapsto (ab\otimes c, -a\otimes bc)$. In particular, for elementary tensors $[\alpha]\otimes [\beta]\otimes [\gamma]$ with $\alpha\beta=dx$, $\beta\gamma=dy$ (and hence $d_1([\alpha]\otimes [\beta]\otimes [\gamma])=0$, the second differential $d_2([\alpha]\otimes [\beta]\otimes [\gamma])=[x\gamma - (-1)^{|\alpha|} \alpha y]$ is given by the formula for the ordinary triple Massey product.
		In general, for arbitrary $n$, the longest possible differential will (on elementary tensors) be given by the formula for $n+2$--fold Massey products.
	\end{ex}
	Motivated by this, we define:
	\begin{defi}
		Let $n\geq 2$. An \textbf{$n$--fold Massey product in the spectral sequence sense} (with respect to $S$) in the augmented algebra $A$ is an element in the image of the map
		\[
		d_{n-1}: EM(A,\Delta^{n-2}, S)_{n-1}^{n-2,\ast}\to EM(A,\Delta^{n-2},S)_{n-1}^{-1,\ast-n+2}.
		\]
	\end{defi}
	
	\begin{rem}\label{rem: functoriality spectral sequence MP}
	    Note that by definition, such a Massey product is a subset of $H(S(A^+))$ (namely, a coset for the space generated by the images of all lower page differentials). By Remark \ref{functoriality}, Massey products can be pulled back and are invariants of an augmented cbba up to augmented weak equivalence.
	\end{rem}

    \begin{rem}
        If there is a weak equivalence $\varphi:B\to A$ from a connected cbba $B$ (i.e. $B^{<0}=0$, $B^0=\C$), then the Massey products do not depend on the augmentation. In fact, $B$ admits the canonical augmentation $B\to B/B^{> 0}\cong \C$ and $\varphi$ will automatically be augmented for any augmentation on $A$. This is the case for the forms on holomorphically simply connected manifolds and on (most, and conjecturally all) nilmanifolds \cite[Theorem 2.34]{Ste23}, \cite{R11}. We do not know if there is an example where they actually depend on the augmentation.
    \end{rem}
    
	\begin{ex}
		Forgetting the $\del$-differential in $A$, i.e. taking $S$ to be the column complex functor, we obtain a spectral sequence version of the Dolbeault--Massey products \cite{CT15} (which are the natural Massey products in Neisendorfer--Taylor's theory \cite{NT78}).
	\end{ex}

As a more substantial example, we will now generalize ABC--Massey triple products to arbitrary length. Our generalization will be of a different nature than the one considered by Tardini \cite[\textsection 4.4]{Ta17}, where an odd number of Bott--Chern classes is taken in to produce an Aeppli class. 
	
To this end, we recall \cite[4.b]{S07} the Schweitzer complex $S_{p,q}(A)$ associated to a double complex $A$ and a fixed bidegree $(p,q)$, which presents $H_A^{p-1,q-1}$ and $H_{BC}^{p,q}$ as the cohomology of a singly-graded complex; the cohomologies of this complex define cohomological functors \cite{Ste22}. The complex is given by 
$$
{\small \begin{tikzcd}
                   &                                                  &                                        &                           &                                             & {\ \ \ \ \ \ \ \ } \\
                   &                                                  &                                        &                           & {A^{p,q+1}\oplus A^{p+1,q}} \arrow[ru, "d"] &    \\
                   &                                                  &                                        & {\  \ \ A^{p,q} \  } \arrow[ru, "d"] &                                             &    \\
                   &                                                  & {A^{p-1,q-1}} \arrow[ru, "\deldelbar"] &                           &                                             &    \\
                   & {A^{p-2,q-1} \oplus A^{p-1,q-2}} \arrow[ru, "d"] &                                        &                           &                                             &    \\
{} \arrow[ru, "d"] &                                                  &                                        &                           &                                             &   
\end{tikzcd} }$$
	where $d$ denotes the map induced by restricting or projecting the de Rham differential in the obvious way. 
	
	\begin{figure}[h!] \centering \scalebox{0.3}{
			\includegraphics[width=\textwidth]{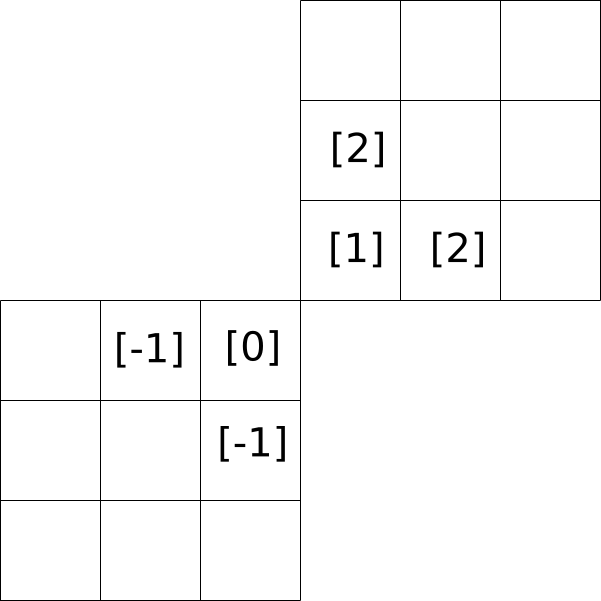}}
			\caption{The indices in the Schweitzer complex.}
	\end{figure}
	
	Note that the cohomology of this complex at the $A^{p,q}$ entry is precisely $H_{BC}^{p,q}(A)$, and at $A^{p-1,q-1}$ it is $H_A^{p-1,q-1}(A)$. Indexing the entries so that $A^{p-1,q-1}$ is at index 0 and $A^{p,q}$ is at index 1, let us denote by $H_{S_{p,q}}^i(A)$ the cohomology at index $i$; in this convention, $H_A^{p-1,q-1}(A) = H_{S_{p,q}}^0(A)$, $H_{BC}^{p,q}(A) = H_{S_{p,q}}^1(A)$ and
		\[
	H_{S_{p,q}}^{-1}(A)=\frac{\ker \left(\pr\circ~d: A^{p-2,q-1}\oplus A^{p-1,q-2}\longrightarrow A^{p-1,q-1}\right)}{\im \left(\pr\circ~ d:A^{p-3,q-1}\oplus A^{p-2,q-2}\oplus A^{p-1,q-3}\longrightarrow A^{p-2,q-1}\oplus A^{p-1,q-2}\right)}.
	\]

	Let us explain how triple ABC--Massey products arise in this setup. The part of the first page of the Eilenberg--Moore spectral sequence $EM(A,\Delta^1,S_{p,q})$ relevant for us looks as follows: 
	$$
	\begin{tikzcd}
		     {H_{BC}^{p,q}(A^+)}& {H_{BC}^{p,q}((A^+)^{\otimes 2})\oplus H_{BC}^{p,q}((A^+)^{\otimes 2})} \arrow[l]       & {H_{BC}^{p,q}((A^+)^{\otimes 3})}\arrow[l]    \\
		 {H_{A}^{p-1,q-1}(A^+)}& {H_{A}^{p-1,q-1}((A^+)^{\otimes 2})\oplus H_{A}^{p-1,q-1}((A^+)^{\otimes 2})} \arrow[l] & {H_{A}^{p-1,q-1}((A^+)^{\otimes 3})}\arrow[l]
	\end{tikzcd}
	$$
	
	Now pick three Bott--Chern classes $a,b,c$ whose degrees sum up to $(p,q)$. There is a natural map $\left( H_{BC}(A^+)\right)^{\otimes 3}\to H_{BC}((A^+)^{\otimes 3})$, so their tensor product defines an element in $H_{BC}^{p,q}((A^+)^{\otimes 3})$. If $ab = bc = 0$, this element will be $d_1$-closed. The second differential, i.e. the triple ABC--Massey product is represented by a zigzag:

$$	
\begin{tikzcd}
{S_{p,q}^1(A^+)} & {S_{p,q}^1((A^+)^{\otimes 2})\oplus S_{p,q}^1((A^+)^{\otimes 2})}                                   & {S_{p,q}^1((A^+)^{\otimes 3})} \arrow[l] \\
{S_{p,q}^0(A^+)} & {S_{p,q}^0((A^+)^{\otimes 2})\oplus S_{p,q}^0((A^+)^{\otimes 2})} \arrow[u, "\deldelbar"] \arrow[l] & {S_{p,q}^0((A^+)^{\otimes 3})}          
\end{tikzcd}$$
	
	If $a=[\alpha], b=[\beta], c=[\gamma]$ and $ab=[\deldelbar x], \beta\gamma=\deldelbar y$, such a zigzag is given by
	
	$$
\begin{tikzcd}
                   & {(\alpha \beta \otimes \gamma, -\alpha \otimes \beta \gamma)} & \alpha \otimes \beta \otimes \gamma \arrow[l] \\
x\gamma - \alpha y & {(x \otimes \gamma, -\alpha \otimes y)} \arrow[u] \arrow[l]   &                                              
\end{tikzcd}
	$$ 
	So $d_2(a\otimes b\otimes c)$ is represented by the usual triple ABC--Massey product.
	
Now we describe the quadruple ABC--Massey products. They will be represented by a zigzag: 
	$$
\begin{tikzcd}
{S_{p,q}^1(A^+)} & {(S_{p,q}^1((A^+)^{\otimes 2})^{\oplus 3}}                     & {(S_{p,q}^1((A^+)^{\otimes 3})^{\oplus 3}}                     & {S_{p,q}^1((A^+)^{\otimes 4})} \arrow[l] \\
{S_{p,q}^0(A^+)} & {(S_{p,q}^0((A^+)^{\otimes 2})^{\oplus 3}}                     & {(S_{p,q}^0((A^+)^{\otimes 3})^{\oplus 3}} \arrow[u] \arrow[l] & {S_{p,q}^0((A^+)^{\otimes 4})}           \\
{S_{p,q}^{-1}(A^+)} & {(S_{p,q}^{-1}((A^+)^{\otimes 2})^{\oplus 3}} \arrow[u] \arrow[l] & {(S_{p,q}^{-1}((A^+)^{\otimes 3})^{\oplus 3}}                     & {S_{p,q}^{-1}((A^+)^{\otimes 4})}          
\end{tikzcd}
	$$
	
	If we start with classes $a,b,c,d \in H_{BC}$ of pure bidegree, with representatives $\alpha, \beta, \gamma, \delta$ such that $|\alpha\beta\gamma\delta|=(p,q)$ and assume there are pure bidegree elements $x,y,z,\eta,\eta',\xi,\xi'$ such that 
	\begin{equation}\label{defining system}
	\deldelbar x=\alpha\beta,~\deldelbar y=\beta\gamma,~\deldelbar z=\gamma\delta\text{ and }x \gamma - \alpha y = \del \eta + \delbar \eta',~y \delta - \beta z = \del \xi + \delbar \xi'
	\end{equation}
	such a zigzag may be represented as
	$$
		{\tiny
\begin{tikzcd}
                                                                                                                                       &                                                                                                                                                                                                    & \begin{pmatrix} \alpha \beta \otimes \gamma \otimes \delta \\ -\alpha \otimes \beta \gamma \otimes \delta \\ \alpha \otimes \beta \otimes \gamma \delta\end{pmatrix} & \alpha \otimes \beta \otimes \gamma \otimes \delta \arrow[l] \\
                                                                                                                                       & \begin{pmatrix} \alpha \otimes (y\delta - \beta z) \\ \alpha \beta \otimes z - x \otimes \gamma \delta \\ (x \gamma - \alpha y) \otimes \delta \end{pmatrix}                                            & \begin{pmatrix} x \otimes \gamma \otimes \delta \\ - \alpha \otimes y \otimes \delta \\ \alpha \otimes \beta \otimes z\end{pmatrix} \arrow[l] \arrow[u]              &                                                              \\
\begin{tabular}{c} $(-1)^{|\alpha|} \alpha(\xi + \xi')$ \\ $-(\del x)z + (-1)^{|x|+1} x \delbar z$ \\ $+ (\eta + \eta')\delta$ \end{tabular} & \begin{pmatrix} (-1)^{|\alpha|} \alpha \otimes (\xi + \xi') \\ - \left( \del x \otimes z + (-1)^{|x|} x \otimes \delbar z \right) \\ (\eta + \eta')\otimes \delta \end{pmatrix} \arrow[l] \arrow[u] &                                                                                                                                                                      &                                                             
\end{tikzcd}
		}
	$$

	The quadruple ABC--Massey product of $a,b,c,d$ is then represented by the class 
	\[[(-1)^{|\alpha|} \alpha(\xi + \xi') -(\del x)z + (-1)^{|x|+1} x \delbar z + (\eta + \eta')\delta]
	\]in (a quotient of) $H_{S_{p,q}}^{-1}(A^+)$. One can continue this construction to obtain formulas for the ``higher'' ABC--Massey products.

	\begin{rem}\label{bottleneck} Note that we could have taken our initial data in a different cohomology group of the Schweitzer complex of a tensor power of $A^+$, beside Bott--Chern cohomology. Hence we in fact obtain a doubly--indexed family of Massey--like products in the spectral sequence sense; one index for the ``length'' of the product (that we may think of as the number of inputs), and one index for the cohomology group of the Schweitzer complex our inital data lives in. 
	
	Three somewhat distinct behaviors emerge among the doubly--indexed family of products in the spectral sequence sense: The upper ``half'' of the Schweitzer complex (i.e. with index $\geq 1$) has a map of complexes to the total de Rham complex, while the lower half (i.e. with index $\leq 0$) receives a map from the total de Rham complex. Consequently, if we start in high degree in the upper half of the Schweitzer complex and form sufficiently short Massey products, these products map to the ordinary de Rham Massey products (but they live in finer groups). Similarly, we may map de Rham Massey products to ABC--Massey products starting in the lower part of the Schweitzer complex. 
	
	On the other hand, those products that cross the \textbf{$\deldelbar$-bottleneck} in the Schweitzer complex, i.e. with input in the upper half and output in the lower half, are genuinely new phenomena that do not seem to be related in any straightforward way to ordinary Massey products. 
	
	We also remark that one can map Bott--Chern classes to the cohomology at any index in the appropriate Schweitzer complex, and hence, if one prefers to start with a pure tensor of Bott-Chern classes, this is possible at any stage. \end{rem}

	\subsection{Ad hoc Massey products}
	The ABC--Massey products introduced above are associated with an augmented cbba. This is in contrast with the more common ad hoc definitions of (triple ABC--) Massey products, which do not need an augmentation, as for instance in \cite[Section 2]{M58}, \cite{K66} (resp. \cite{AT15}, \cite{Ta17}). The ad hoc version always starts with a pure tensor of classes as input and outputs a subset of cohomology, which, apart from the triple product case, is generally not the coset of a linear subspace. The product is then said to vanish if zero is contained in this subset. 
	
	There is, in principle, a straightforward way of obtaining such an unaugmented ad hoc version from the explicit description of the differentials in the spectral sequence version. Since the amount of necessary notation becomes unwieldy, let us only sketch this for triple and quadruple products, taking our cues from our examples discussed in the previous section and \cite{K66}:
	
	\begin{defi}\label{adhocquadruple}
	Let $A$ be a cbba and $a,b,c,d\in H_{BC}(A)$ classes of pure bidegree with representatives $\alpha,\beta,\gamma,\delta$ and write $(p_\alpha,q_\alpha)=|\alpha|$, etc.
	\begin{enumerate}
	    \item A \textbf{defining system} for the ad hoc ABC--Massey quadruple product $\langle a,b,c,d\rangle$ is a collection of  elements $x,y,z,\tilde{\eta},\tilde{\xi}$ such that
	    \begin{align*}
	    x\in S_{|\alpha\beta|}^{0}(A),~ y\in S_{|\beta\gamma|}^{0}(A),~ z\in S_{|\gamma\delta|}^0(A),\\
	    \tilde{\eta}=\eta+\eta'\in S_{|\alpha\beta\gamma|}^{-1}(A)=A^{|\alpha\beta\gamma|-(2,1)}\oplus A^{|\alpha\beta\gamma|-(1,2)},\\
	    \tilde{\xi}=\xi+\xi'\in S_{|\beta \gamma\delta|}^{-1}(A)=A^{|\beta \gamma\delta|-(2,1)}\oplus A^{|\beta \gamma\delta|-(1,2)},
	    \end{align*}
	    satisfying \Cref{defining system} above.
	    \item The \textbf{ad hoc quadruple ABC--Massey product} $\langle a,b,c,d\rangle$ is the subset of $H_{S_{|\alpha\beta\gamma\delta|}}^{-1}(A)$ given by the collection of all classes
	    \[
	    [(-1)^{|\alpha|} \alpha\tilde{\xi} -(\del x)z + (-1)^{|x|+1} x \delbar z + \tilde{\eta}\delta]
	    \]
	    determined by a defining system.
	\end{enumerate}  
	\end{defi}
	
Another choice for the classes in the ad hoc quadruple product would be $$[(-1)^{|\alpha|} \alpha\tilde{\xi} + (\delbar x)z + (-1)^{|x|} x \del z + \tilde{\eta}\delta].$$ Note that the difference between the representatives of this and the previous choice is $d(xz)$, hence the two representatives give the same class in $H_{S_{|\alpha\beta\gamma\delta|}}^{-1}(A)$.
\begin{rem}
    As seen above, if one makes the analogous definitions in order to define an ad hoc triple ABC--Massey product, one recovers the ABC--Massey product $\langle a,b,c\rangle$ of \cite{AT15}. One can also pursue defining ad hoc quintuple and higher products, but we do not do so here.
\end{rem}

The functoriality and invariance properties in this ad hoc setup need to be proven with more care:

\begin{prop}\label{prop: MPs E1-iso stable}
	Let $\varphi:A\to B$ be a map of cbba's. 
		\begin{enumerate}
			\item (Functoriality) For $a,b,c\in H_{BC}(A)$, one has $\varphi \langle a,b,c\rangle\subseteq \langle{\varphi(a),\varphi(b),\varphi(c)}\rangle$. 
			\item (Invariance under weak equivalences) If $\varphi$ is a weak equivalence, $a,b,c\in H_{BC}(A)$, $a',b',c'\in H_{BC}(B)$, there are equalities of sets $$\varphi \langle a,b,c\rangle=\langle \varphi(a),\varphi(b),\varphi(c)\rangle\subseteq H_A(B)$$
			and 
			$$\varphi^{-1}\langle a',b',c'\rangle=\langle\varphi^{-1}(a'),\varphi^{-1}(b'),\varphi^{-1}(c')\rangle\subseteq H_A(A).$$
		\end{enumerate} 
		In particular, admitting a non-trivial Massey product is invariant under weak equivalences.
	\end{prop}

	\begin{proof}
	Given any defining system for the $\langle a,b,c\rangle$, its image under $\varphi$ is a defining system for $\langle{\varphi(a),\varphi(b),\varphi(c)}\rangle$, hence the functoriality assertion.  For the first equality of part $(2)$, we have to show the other inclusion. Let $a=[\alpha],b=[\beta],c=[\gamma]$ and $x,y\in B$ with $\deldelbar x=\varphi(\alpha\beta)$ and $\deldelbar y=\varphi(\beta\gamma)$, such  that $[\varphi(\alpha)y-x\varphi(\gamma)]\in \langle\varphi(a),\varphi(b),\varphi(c)\rangle$. Then we want to find $\varphi(x'),\varphi(y')\in \im \varphi$ such that $\deldelbar x'=\alpha\beta$ and $\deldelbar y'=\beta\gamma$ such that $[\varphi(\alpha)y-x\varphi(\gamma)]=[\varphi(\alpha y')-\varphi(x'\gamma)]\in \varphi\langle a,b,c\rangle$. First, since $H_{BC}(A)\cong H_{BC}(B)$, we may pick \textit{some} primitives $\deldelbar\varphi(\tilde{x})=\varphi(\alpha\beta)$ and $\deldelbar(\tilde{y})=\varphi(\beta\gamma)$. Now, $\varphi(\tilde{x})-x$ and $\varphi(\tilde{y})-y$ are $\deldelbar$-closed and therefore define Aeppli-classes. Because $\varphi$ is a weak equivalence, we may choose elements $\tilde{x}',\tilde{y}'\in A$ such that $[x-\varphi(\tilde{x})]=[\varphi(\tilde{x}')]$ and $ [y-\varphi(\tilde{y})]=[\varphi(\tilde{y}')]$. Then, setting $x'= \tilde{x}'+\tilde{x}$ and $y'= \tilde{y}'+\tilde{y}$ does the job, as
		\[
		[\varphi(\alpha)y-x\varphi(\gamma)]=[\varphi(\alpha)(y - \varphi(\tilde{y}) + \varphi(\tilde{y})) - (x - \varphi(\tilde{x}) + \varphi(\tilde{x}))\varphi(\gamma)]=[\varphi(\alpha)\varphi(y')-\varphi(x')\varphi(\gamma)].
		\]
		The second assertion follows from the first since we may write $a'=\varphi(a),b'=\varphi(b),c'=\varphi(c)$ for some $a,b,c\in H(A)$.
	\end{proof}
	
	\begin{rem}\label{remarkafterprop}
		The same kind of arguments (with more cumbersome notation) show that Proposition \ref{prop: MPs E1-iso stable} remains valid for quadruple ABC--Massey products, for Tardini's \cite{Ta17} Massey products, and also for usual Massey products and quasi-isomorphisms.
	\end{rem}

	\begin{rem}
	    We would like to emphasize that neither invariant, the Massey products in the spectral sequence sense or the ad hoc sense, is finer than the other. The spectral sequence products have a large indeterminacy (e.g. \emph{any} product of positive-degree cohomology classes will represent a trivial triple product in the spectral sequence sense), but also a larger domain of definition, as the input data is not restricted to pure tensors. One can find a discussion on this (for ordinary Massey products) in \cite{P17}.
	\end{rem}
	
	\begin{ex}\label{quadruple} We exhibit an example of a non-vanishing ad hoc quadruple ABC--Massey product, with four Bott--Chern classes as input. Consider a complex nilmanifold with model $$\left( \Lambda(x, \bar{x}, y, \bar{y}, z, \bar{z}, w, \bar{w}), dz = xy, dw = xz \right).$$ That is, the above is the algebra of invariant forms on a complex parallelizable nilmanifold together with its bigrading; the inclusion of this into all forms is a weak equivalence by \cite{Sa76}; see also the work of Console--Fino, Rollenske and others, summarized in \cite{R11}. 
		
		Consider the Bott--Chern classes represented (uniquely) by $x, x\bar{y}, y\bar{x}$. Note that $x^2 = 0$, $x(x\bar{y}) = 0$, and that $(x\bar{y})(\bar{x}y) = \deldelbar(z\bar{z})$. The triple ABC--Massey product $\langle x, x\bar{y}, \bar{x}y\rangle$ vanishes, as we have $xz\bar{z} = \del(w\bar{z})$. The triple product $\langle x, x, x\bar{y} \rangle$ also vanishes, trivially.
		
		The quadruple ABC--Massey product $\langle x, x, x\bar{y}, \bar{x}y \rangle$ is well-defined, and is represented by the $H_{S_{4,2}}^{-1}$ class $[xw\bar{z}]$. To verify that this quadruple product is indeed non-vanishing, we show that no other choice of primitive elements can yield the trivial class. First of all, notice that 0 is the only possible choice of $\deldelbar$-primitive for $x^2$ and for $x(x\bar{y})$. A primitive for $(x\bar{y})(\bar{x}y)$ can be anything in the affine space $$z\bar{z} + \mathrm{span}\left(x\bar{x}, x\bar{y}, x\bar{z}, x\bar{w}, y\bar{x}, y\bar{y}, y\bar{z}, y\bar{w}, z\bar{x}, z\bar{y}, w\bar{x}, w\bar{y}\right).$$ Therefore the possible representatives for $\langle x, x\bar{y}, \bar{x}y \rangle$ lie in the affine space $$xz\bar{z} + \mathrm{span}\left( x y \bar{x}, xy\bar{y}, xy\bar{z}, xy\bar{w}, xz\bar{x}, xz\bar{y}, xw\bar{x}, xw\bar{y} \right).$$ Since $\delbar$ vanishes on $(2,0)$ forms, any representative of the quadruple ABC--Massey product is thus in \[x\left( w\bar{z} + \mathrm{span}\left( x\bar{x}, x\bar{y}, x\bar{z}, x\bar{w}, y\bar{x}, y\bar{y}, y\bar{z}, y\bar{w}, z\bar{x}, z\bar{y}, w\bar{x}, w\bar{y}, z\bar{z}, z\bar{w} \right) \right).\] No differential from $S^{-2}_{4,2}$ contains a term of $xw\bar{z}$, and hence we are done. 

This manifold is a holomorphic analogue of the filiform nilmanifold, which is the simplest example of a nilmanifold with a non-vanishing quadruple Massey product. Indeed, our complex nilmanifold is a holomorphic torus bundle over the Iwasawa bundle, while the filiform nilmanifold is a circle bundle over the Heisenberg nilmanifold. (Here we refer to ``the'' filiform and Heisenberg nilmanifold, though there are various lattices in the appropriate simply connected nilpotent Lie group one could choose, resulting in different homotopy types; however, the rational homotopy theoretic minimal model is unique.)	
\end{ex} 
	
\begin{rem}
	Our definitions of ABC--Massey products can be adapted to the almost complex setting by using the double complexes $(A_X)_s$ and $(A_X)_q$, associated with an almost complex manifold $X$ in \cite{CPS22}, and the ensuing notions of Schweitzer complexes.
\end{rem}

	\section{Bigraded notions of formality}
	
	In usual rational homotopy theory, a commutative differential graded algebra (cdga) is called \textit{formal} if it is connected by a chain of quasi-isomorphisms to a cdga with trivial differential (whose underlying algebra is consequently isomorphic to the cohomology of the original cdga). A singly-graded complex additively splits into dots and lines (i.e. zigzags of length two), where the latter do not contribute to cohomology. Hence we may interpret formality of a cdga as the existence of a chain of quasi-isomorphisms to a cdga with either only dots, or (equivalently) with no shapes that do not contribute to cohomology. This motivates the following two notions of formality in the bigraded setting, where we wish to consider the differentials $\del$ and $\delbar$ on equal footing:
	
	\begin{defi}
		A cbba $A$ is called:
		\begin{enumerate}
			\item \textbf{weakly formal}, if it is connected by a chain of weak equivalences to a cbba $H$ which satisfies $\del_H\delbar_H=0$.
			\item \textbf{strongly formal}, if it is connected by a chain of weak equivalences to a cbba $H$ which satisfies $\del_H=\delbar_H=0$.
		\end{enumerate}
		We use the same terminology if $A$ is augmented and the weak equivalences preserve augmentations. 		
	\end{defi}

 \begin{rem}
     The augmented versions of formality imply the unaugmented versions. Conversely, one can show that if $A$ admits a connected cofibrant model in the sense of \cite{Ste23}, the unaugmented versions imply the augmented ones.
 \end{rem}
	A strongly formal cbba is one that is connected by a chain of weak equivalences to a cbba whose underlying complex consists only of dots, while a weakly formal one is connected to a cbba which additively has no squares.
	
	Note that strong formality implies weak formality and the $\deldelbar$-lemma (the latter being an (additive) bigraded quasi--isomorphism invariant of double complexes). In the definition of strong formality, one may thus take $H$ to be $H_{BC}(A)$. Conversely if $A$ satisfies the $\deldelbar$-lemma, weak and strong formality are equivalent. 
	
	\begin{rem} The notions of \emph{Dolbeault formality} and \emph{strict (Dolbeault) formality} considered in \cite[p.187]{NT78} are implied by our notion of strong formality, and a priori unrelated to weak formality. In fact, we show in  \Cref{prop: strform} that for compact complex manifolds, strict formality is equivalent to the $\deldelbar$-lemma. Through the examples discussed below we see that the $\deldelbar$-lemma does not imply strong or weak formality, nor does weak formality imply the $\deldelbar$-lemma. \end{rem}
 
	
	As an immediate consequence of \Cref{rem: functoriality spectral sequence MP}, \Cref{prop: MPs E1-iso stable} and \Cref{remarkafterprop}, we have the following, where we mean Massey products in the spectral sequence sense for augmented formality and in the ad hoc sense for unaugmented formality:

	\begin{prop}\label{weaklyformalABC--Massey} On a strongly formal manifold all ABC--Massey products vanish. On a weakly formal manifold, all ABC--Massey products that cross the $\deldelbar$-bottleneck (see \Cref{bottleneck}) vanish.\end{prop}
	
	In ordinary rational homotopy theory, and in the Dolbeault homotopy theory of \cite{NT78}, formality is equivalent to augmented formality for cohomologically connected spaces (e.g. compact complex manifolds in the latter), essentially due to the fact that the construction of a cofibrant model does not involve adding generators in degree 0 (or bidegree $(0,0)$ for the latter). From now on, we will consider only the unaugmented notions of formality, unless explicitly stated otherwise.

	Taking metrics into account, recall that a compact Riemannian manifold is called geometrically formal if products of harmonic forms are again harmonic. Considering the Bott--Chern and Aeppli harmonic forms (see \cite[Section 2]{S07}), which we denote by $\cH_{BC}$ and $\cH_{A}$, respectively, one can ask for the additional condition that the collection of harmonic forms is closed under the differential:
	\begin{defi}\label{ABCgeometricallyformal}
		A connected compact complex manifold is called \textbf{ABC-geometrically formal} if it admits a Hermitian metric such that following two equivalent conditions are satisfied:
		\begin{enumerate}
			\item\label{geometric formality i.t.o. ABC} The space $\cH_A+\cH_{BC}$ is closed under multiplication and $\del,\delbar$.
			\item\label{geometric formality i.t.o minimal subbba} There exists a bigraded, bidifferential subalgebra $\cH\subseteq A_X$ which is closed under $\ast$ and satisfies $\deldelbar=0$, such that the inclusion is a weak equivalence.
		\end{enumerate}
	\end{defi}
	Let us prove that these two conditions are indeed equivalent: If condition (\ref{geometric formality i.t.o. ABC}) holds, set $\cH:=\cH_A+\cH_{BC}$. Since
	\[
	\omega\in\cH_{BC}\quad \Leftrightarrow\quad \begin{cases}\del\omega=0\\\delbar\omega=0\\\deldelbar\ast\omega=0\end{cases}\quad\Leftrightarrow\quad \ast\omega\in\cH_{A}, 
	\]
	this space is closed under $\ast$ and satisfies $\deldelbar =  0$ by definition. It is a bigraded subalgebra by assumption. Since $H_{BC}(\cH)=H_{BC}(X)$ and $H_{A}(\cH)=H_A(X)$, the inclusion is a weak equivalence by \Cref{rem: E1 vs ABC quiso}. 
	
	Conversely, if condition (\ref{geometric formality i.t.o minimal subbba}) holds, $\cH$ has to be finite dimensional, since in any decomposition into indecomposable complexes, there can be no squares (since $\deldelbar=0$) and every zigzag has to occur with the same multiplicity as in $A_X$ (where the multiplicities are known to be finite). In particular, $\cH$ is a closed subspace and we have an orthogonal projection $A_X\to \cH$. Because $\cH$ is closed under $\ast$ and is a bigraded subspace, this projection is a map of (bi)complexes: Indeed, pick any orthonormal basis $\{h_i\}$ for $\cH$. Then the projection of any element in $\omega\in A_X$ is written as $\pr_{\cH}(\omega)=\sum\langle \omega, h_i\rangle h_i$. Thus:
	
	\[
	(d\circ \pr_{\cH})(\omega)=d\left(\sum\langle \omega, h_i\rangle h_i\right)=\sum_{i,j}\langle \omega, h_i\rangle \langle d h_i, h_j\rangle h_j
	\]
	\[
	(\pr_{\cH}\circ~d) (\omega)=\sum_i\langle \omega,d^*h_i\rangle h_i=\sum_{i,j}\langle \omega,h_j\rangle \langle d^* h_i,h_j\rangle h_i
	\]
	and both terms coincide. In particular, this implies that the Bott--Chern and Aeppli Laplacians have block-diagonal form with respect to the splitting $
	A_X=\cH\oplus \cH^\perp$
	and consequently $\cH_{BC}=\cH_{BC}\cap \cH \oplus \cH_{BC}\cap \cH^\perp$. Since $H_{BC}(\cH)=H_{BC}(X)$, this implies $\cH_{BC}\subseteq \cH$ and similarly $\cH_A\subseteq \cH$. On the other hand, if $\omega\in \cH$, necessarily $\deldelbar\omega=0$, i.e. it gives rise to an Aeppli class. The harmonic component $\omega^A$ lies in $\cH$, so $\omega'=\omega-\omega^A\in \cH\cap(\im\del+\im\delbar)=\del\cH+\delbar\cH$, where for the last equality we use that an element in $\cH$ which gives the zero class in $H_A(X)$ has already a primitive in $\cH$ as $H_{A}(\cH)\cong H_A(X)$. But then, because $\deldelbar=0$ in $\cH$, we have $\del\omega'=\delbar\omega'=0$ and $\deldelbar\ast\omega'=0$, i.e. $\omega'\in \cH\cap\cH_{BC}$. This completes the proof.

	\begin{rem}  Note that the multiplicative structure was not used in proving the equivalence of the two conditions in \Cref{ABCgeometricallyformal}. So we also have the following statement: the space $\cH_A+\cH_{BC}$ is closed under $\del$ and $\delbar$ if and only if there exists a bigraded subcomplex $\cH\subseteq A_X$ which is closed under $\ast$ and satisfies $\deldelbar=0$, such that the inclusion is an (additive) bigraded quasi--isomorphism. Note also that this equivalence implies that the subspace in the second condition has to coincide with $\cH_{BC}+\cH_A$, and in particular must be real.
	\end{rem}
	
	\begin{rem}
	    Note that the inclusion $\cH\subseteq A_X$ is automatically compatible with the augmentation given by the choice of any basepoint. In particular, ABC-geometric formality implies weak formality in the augmented sense.
	\end{rem}
	
	\begin{rem} (Relation to other notions of geometric formality) ABC-geometric formality implies geometric Bott--Chern formality in the sense of Angella--Tomassini \cite{AT15}, i.e. the Bott--Chern harmonics form an algebra. Indeed, the product of two Bott--Chern harmonic forms will of course be closed under $\del$ and $\delbar$, as well as under $(\deldelbar)^*$ by the assumption of ABC-geometric formality.
		
		If we further assume our manifold satisfies the $\deldelbar$-lemma, then ABC-geometric formality implies Dolbeault geometric formality \cite{TT14} and ordinary geometric formality. Indeed, e.g. for the latter, since the Bott--Chern harmonics and Aeppli harmonics coincide, they are closed under $\del, \delbar, \del^*, \delbar^*$ and hence de Rham harmonic, and vice versa. \end{rem}

	\begin{ex}\label{KT} We show on the example of the Kodaira--Thurston surface $KT$ that ABC-geometric formality, and hence also weak formality, does not imply the $\deldelbar$-lemma nor de Rham formality. 

 Presenting $KT$ as the group of real matrices of the form $$\begin{pmatrix} 1 & u_1 & u_3 & 0 & 0 \\ 0 & 1 & u_2 & 0 & 0 \\ 0 & 0 & 1 & 0 & 0 \\ 0 & 0 & 0 & 1 & u_4 \\ 0 & 0 & 0 & 0 & 1 \end{pmatrix}$$ modulo integer matrices, a direct computation shows that the Lie algebra of left-invariant vector fields is of the form $\mathrm{span}(V_1,V_2,V_3,V_4)$ with the only non-trivial bracket among basis vectors being $[V_1,V_2] = -V_3$. Dualizing we have the rational homotopy theoretic model $(\Lambda(v_1,v_2,v_3,v_4), dv_3 = v_1 v_2)$ of $KT$ given by the cdga of left-invariant one-forms on the above matrix group. A left-invariant complex structure is determined by $J(v_1) = v_2, J(v_3) = v_4$. Complexifying and setting $x = v_1 - iv_2, y = 2i(v_3 - iv_4)$ we have a bigraded model for this complex manifold given by $$\left( \Lambda(x, \bar{x}, y, \bar{y}), dy = x\bar{x} \right),$$ where $x$ and $y$ are in bidegree $(1,0)$. The inclusion of this cbba into the forms on $KT$ is a weak equivalence.
		
		\begin{figure}[h!] \centering \scalebox{0.2}{
				\includegraphics[width=\textwidth]{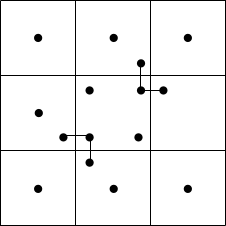}}
			\caption{The double complex of the Kodaira--Thurston surface.}
		\end{figure}
		
		With respect to the obvious (diagonal)  metric, we see that this algebra satisfies the second condition in \Cref{ABCgeometricallyformal}. Being a non-trivial nilmanifold, $KT$ is not formal (nor is it Dolbeault formal). A concrete non-vanishing (Dolbeault--) Massey triple product is $\langle x, \bar{x}, \bar{x}\rangle$, represented by $y \bar{x}$. 
	\end{ex}
	
	\begin{ex}\label{hopf} We consider a Hopf surface $(\CC^2 - \{0\})/ (\alpha \sim \lambda \alpha)$, where $\lambda$ is a non-zero complex number not on the unit circle. This complex manifold is weakly formal, de Rham formal (as the underlying smooth manifold is $S^1 \times S^3$), and does not satisfy the $\deldelbar$-lemma (as the first Betti number is odd). A model, i.e. a cbba with a weak equivalence to the cbba of forms, for the Hopf surface is given by \cite[Lemma 3.10]{Ste23} \[\left( \Lambda(x, \bar{x}, y, z, \del z, \delbar z), dx = -d\bar{x} = y, \deldelbar z = iy^2 \right),\] where $x$ is in bidegree $(1,0)$ and $y,z$ are in $(1,1)$. This model maps to
		\[
		\left(\Lambda(x,\bar{x}, y)/(y^2), dx=-d\bar{x}=y \right)
		\]
		by sending $z,\del z,\delbar z$ to $0$. This map is a weak equivalence and the latter cbba satisfies $\del\delbar\equiv 0$, so the Hopf surface is weakly formal. 
	\end{ex}

 We now compare with the notion of strict formality of \cite{NT78}. Recall from \cite{NT78} that a cbba $A$ is called $\del$-degenerate, if there is a map of cbba's $\varphi:M\to A$ such that $H_{\delbar}(\varphi)$ is an isomorphism and the $(1,0)$-differential $\del$ on $M$ is trivial. A cbba is called strictly formal if it is $\del$-degenerate and formal as a differential bigraded algebra (DBA). While on the surface this seems to be a multiplicative notion, we show below that on compact complex manifolds this depends only on the additive structure of the complex of forms:
 
 \begin{prop}\label{prop: strform}
Let $X$ be a compact complex manifold. The following assertions are equivalent:
\begin{enumerate}
    \item\label{lab: ddbar} $X$ satisfies the $\deldelbar$-lemma.
    \item\label{lab: dclosed} Every Dolbeault cohomology class can be represented a $d$-closed form.
    \item\label{lab: strform} $X$ is strictly formal.
    \item\label{lab: del-deg} $X$ is $\del$-degenerate.
\end{enumerate}
\end{prop}
\begin{proof}
The implication (\ref{lab: ddbar})$\Rightarrow$(\ref{lab: strform}) is shown in \cite{NT78} with an argument from \cite{DGMS75}. The implication (\ref{lab: strform})$\Rightarrow$(\ref{lab: del-deg}) follows from the definition. For the implication (\ref{lab: del-deg})$\Rightarrow$(\ref{lab: dclosed}), consider a map $\varphi: M\to A_X$ as in the definition of $\del$-degeneracy. Note that we can write every Dolbeault cohomology class $\mathfrak{c}\in H_{\delbar}(X)$ as $\mathfrak{c}=[\varphi(m)]$ for some $\delbar$-closed element in $M$. But $\del$ on $M$ is trivial and so $m$, hence $\varphi(m)$ is $d$-closed.

Finally, we show (\ref{lab: dclosed})$\Rightarrow$(\ref{lab: ddbar}). Choose a decomposition of $A_X$ into indecomposable double complexes (`squares and zigzags'). Recall \cite{DGMS75}, \cite{Ste21} that the $\deldelbar$-lemma is equivalent to the statement that only squares and dots appear in such a decomposition, i.e. that zigzags of length $\geq 2$ have multiplicity $0$. By assumption (\ref{lab: dclosed}), all Dolbeault cohomology classes in all zigzags with nonzero multiplicity have to have a closed representative. This rules out zigzags of length $\geq 2$ for which the top left generator is $\delbar$-closed, i.e. anything of the following form:
\[
\begin{tikzcd}
\bullet \arrow[r] & \bullet                              \\
                  & \quad \quad \ddots \arrow[u, dashed]
\end{tikzcd}
\]
Now, on a compact complex $n$-fold, the multiplicity of any zigzag $Z$ is the same as that of its reflection along the diagonal $\sigma Z$ and as that of its reflection along the $n$-th antidiagonal $\tau Z$, see \cite{Ste21}. But for any zigzag which is not a dot, there is a zigzag in the $\langle \sigma,\tau\rangle$-orbit which is ruled out by the above argument.
\end{proof}

\begin{rem}
    Note that one could assume even less than $\del$-degeneracy: Namely, it is sufficient to require that there is a map of double complexes (not necessarily multiplicative) $\varphi:M\to A_X$ where $M$ has $\del\equiv 0$ and $H_{\delbar}(\varphi)$ is surjective.
\end{rem}

\begin{rem}
 In \cite{NT78} it is claimed that the Hopf surface is strictly formal. This contradicts the above result since it does not satisfy the $\deldelbar$-lemma. Or, just using (\ref{lab: dclosed}), one may see that a generator for $H_{\delbar}^{0,1}$ in degree $(0,1)$ cannot have a closed representative, since otherwise its conjugate would be a nonzero holomorphic $1$-form, of which there are none.
\end{rem}

 \begin{ex} In \cite{ST23} the authors consider six--dimensional nilmanifolds with a left-invariant complex structure admitting a left-invariant SKT metric. As is immediate from the proof of \cite[Theorem 7.2]{ST23}, these are weakly formal (and the non-abelian ones are not strongly formal). \end{ex}
	
		\begin{ex} Consider the Calabi--Eckmann complex structure on $S^3 \times S^3$, with a global basis $\{\phi^1, \phi^2, \phi^3 \}$ of $(1,0)$--forms satisfying \begin{align*} d \phi^1 &= i \phi^1 \phi^3 + i \phi^1 \overline{\phi^3}, \\ d \phi^2 &= \phi^2 \phi^3 - \phi^2 \overline{\phi^3}, \\ d \phi^3 &= -i \phi^1 \overline{\phi^1} + \phi^2 \overline{\phi^2}, \end{align*} see \cite[\textsection 3]{TT17}. Note that the inclusion of the algebra generated by these elements and their conjugates into all forms is an injection on $H_{\delbar}$ since it is a map of algebras satisfying Serre duality, with respect to the invariant volume form $\phi^1\phi^2\phi^3\overline{\phi^1}\overline{\phi^2}\overline{\phi^3}$. From knowledge of the Hodge numbers of $S^3\times S^3$ \cite{B78}, the map is also a surjection on $H_{\delbar}$, and hence, since it preserves real structures (see \Cref{realstructure}), it is a weak equivalence. With respect to the diagonal  metric, the subalgebra generated by $\phi^1 \overline{\phi^1}, \phi^2 \overline{\phi^2}, \phi^3 \overline{\phi^3}$ and their $\del$ and $\delbar$-derivatives satisfies the second condition of \Cref{ABCgeometricallyformal}, and so in particular this manifold is weakly formal. 
		
		As shown in \cite[\textsection 3]{TT17}, there is a small deformation of this complex manifold that carries a non-trivial ABC--Massey triple product. In particular, by \Cref{weaklyformalABC--Massey}, weak formality is not stable under (small) deformation. 
		
	\end{ex}
	
	\begin{ex} We give a non-trivial example of a strongly formal manifold. Consider the full flag manifold $SU(3)/\left( U(1) \times U(1)\right)$. 

\begin{table}[h]
\begin{tabular}{c|c|c|c|c|c|c|c|c|}
 & $h_1$ & $h_2$ & $h_3$ & $h_4$ & $h_5$ & $h_6$ & $h_7$ & $h_8$ \\ \hline

$h_1$  &  & 0 & $-2h_4$ & $2h_3$ & $-h_6$ & $h_5$ & $h_8$ & $-h_7$ \\ \hline
$h_2$  &  &  & $h_4$ & $-h_3$ & $-h_6$ & $h_5$ & $-2h_8$ & $2h_7$ \\ \hline
$h_3$  &  &  &  & $-2h_1$ & $-h_8$ & $h_7$ & $-h_6$ & $h_5$ \\ \hline
$h_4$   &  &  &  &  & $-h_7$ & $-h_8$ & $h_5$ & $h_6$ \\ \hline
$h_5$   &  &  &  &  &  & $-2h_1-2h_2$ & $-h_4$ & $-h_3$ \\ \hline
$h_6$ &   &  &  &  &  &  & $h_3$ & $-h_4$ \\ \hline
$h_7$ &  &  &  &  &  &  &  & $-2h_2$ \\ \hline
$h_8$ &   &  &  &  &  &  &  &  \\ \hline
\end{tabular}
\vspace{1em}
\caption{Lie brackets for $\su(3)$}
\end{table}
The structure coefficients for the Lie algebra $\su(3)$ are given in Table 1, in the basis $\{h_i\}$ explicitly described in \cite[Section 7]{CZ}. We can take the span of $h_1, h_2$ to be the subalgebra corresponding to $\mathfrak{u}(1) \oplus \mathfrak{u}(1)$. The relative complex $C(\su(3), \mathfrak{u}(1) \oplus \mathfrak{u}(1))$ (i.e. the complex of cochains orthogonal to $\mathfrak{u}(1) \oplus \mathfrak{u}(1)$ in the terminology of \cite[Section 22]{CE}) is spanned in positive degrees by \\
\begin{itemize} \item[] $ \alpha_2:= h^3h^4, \alpha_2':= h^5h^6, \alpha_2'':= h^7h^8$ in degree 2, \\ \item[] $\alpha_3:= h^3h^5h^7 + h^3h^6h^8 - h^4h^5h^8 + h^4h^6h^7, \alpha_3':= h^3h^5h^8 - h^3h^6h^7 + h^4h^5h^7 + h^4h^6h^8$ in degree 3, \\ \item[] $\alpha_2\alpha_2' = h^3h^4h^5h^6, \alpha_2\alpha_2'' = h^3h^4h^7h^8, \alpha_2'\alpha_2'' = h^5h^6h^7h^8$ in degree 4, \\ \item[] $\alpha_2\alpha_2'\alpha_2'' = h^3h^4h^5h^6h^7h^8$ in degree 6. \\ \end{itemize} The above was obtained using \cite{Maple}. The almost complex structure determined by $Jh_1 = -h_2, Jh_3 = -h_4, Jh_5 = -h_6, Jh_7 = -h_8$ induces a left-invariant $J$ on $SU(3)/(U(1)\times U(1))$ which is integrable, as follows from a direct computation of the Nijenhuis tensor. Taking the basis $\phi_1 = h^1 - ih^2, \phi_2 = h^3 - ih^4, \phi_3 = h^5 - ih^6, \phi_4 = h^7 - ih^8$ for $\Lambda^{1,0}(\su(3)^* \otimes \CC)$, we have \begin{align*} \alpha_2 &= -\tfrac{i}{2} \phi^2 \overline{\phi^2}, \,\, \alpha_2' = -\tfrac{i}{2} \phi^3 \overline{\phi^3}, \,\, \alpha_2'' = -\tfrac{i}{2} \phi^4 \overline{\phi^4}, \\ \alpha_3 &= -\tfrac{1}{2} \left( \phi^2 \phi^4 \overline{\phi^3} + \phi^3 \overline{\phi^2} \overline{\phi^4} \right), \,\, \alpha_3' = \tfrac{i}{2} \left( \phi^3 \overline{\phi^2} \overline{\phi^4} - \phi^2 \phi^4 \overline{\phi^3} \right) . \end{align*}

Note that $\beta = -\alpha_3 + i\alpha_3'$ is of bidegree $(2,1)$, and $\{\beta, \overline{\beta}\}$ spans the degree 3 part of the relative complex. We calculate $d\alpha_2' = -d\alpha_2 = -d\alpha_2'' = \tfrac{1}{2} \left( \beta + \overline{\beta} \right)$, $d\alpha_3 = 0$, and $d\alpha_3' = 4(\alpha_2\alpha_2'' - \alpha_2\alpha_2' - \alpha_2'\alpha_2'')$. Setting $\gamma = 4i(\alpha_2\alpha_2'' - \alpha_2\alpha_2' - \alpha_2'\alpha_2'')$, we thus have $d\beta = \overline{\partial}\beta = \gamma$, and the relative complex is given by the following double complex:


\begin{center}
			\begin{figure}[h!]\label{su3u1u1}
				\begin{tikzcd}[column sep = tiny, row sep = small]
					&                              &                                          &  &                                       &                                       & \bullet_{\alpha_6} \\
					&                              &                                          &  & \bullet_{(\alpha_2+\alpha_2')^2} &                                       &                    \\
					&                              & \bullet_{\overline{\beta}} \arrow[rr]    &  & \bullet_{\gamma}                      & \bullet_{(\alpha_2-\alpha_2'')^2} &                    \\
					&                              &                                          &  &                                       &                                       &                    \\
					& \bullet_{\alpha_2+\alpha_2'} & \bullet_{2\alpha_2} \arrow[uu] \arrow[rr] &  & \bullet_{\beta} \arrow[uu]            &                                       &                    \\
					&                              & \bullet_{\alpha_2 - \alpha_2''}          &  &                                       &                                       &                    \\
					\bullet_1 &                              &                                          &  &                                       &                                       &                   
				\end{tikzcd}
				\caption{The left-invariant double complex for $SU(3)/\left( U(1) \times U(1) \right)$}
			\end{figure}
		\end{center}

 
		

		We can model this complex by $$\left( \Lambda(x, y, z, w ), \deldelbar z = xy - x^2 - y^2, \deldelbar w = x^3 \right),$$ where $x,y,z$ are in degree $(1,1)$ and $w$ is in $(2,2)$, by mapping $$x \mapsto \alpha_2 + \alpha_2', \ y \mapsto \alpha_2 - \alpha_2'', \ z \mapsto \alpha_2, \ w \mapsto 0.$$ This map is a weak equivalence. On the other hand, the inclusion of the full complex above into the complex of all forms on the flag manifold is also a weak equivalence. Indeed, $(\alpha_2 + \alpha_2')+(\alpha_2 - \alpha_2'')$ yields a left-invariant K\"ahler form; by a classical result of Chevalley--Eilenberg \cite[Theorem 22.1]{CE}, the inclusion of the complex into all forms induces an isomorphism on de Rham cohomology, so by the degeneracy of the Fr\"olicher spectral sequence our claim follows.
		
		Hence our manifold is weakly formal. Since it also satisfies the $\deldelbar$-lemma, it is equivalently strongly formal. Notice that it is not ABC--geometrically formal with respect to any $SU(3)$-invariant metric, as otherwise the products of its degree two part would span a three--dimensional space.

	\end{ex}
	
		\begin{ex}
	 Sferruzza and Tomassini have constructed an example of a compact complex threefold which satisfies the $\deldelbar$-lemma (and hence has no non-trivial Dolbeault--Massey products), yet carries a non-vanishing triple ABC--Massey product \cite{ST22}, and hence is not weakly formal.  For the convenience of the reader, we give a simple algebraic version of their example: Consider the cbba of left-invariant forms on the Iwasawa manifold:
	\[
	\left(\Lambda(\varphi^1,\varphi^2,\varphi^3,\bar\varphi^1,\bar\varphi^2,\bar\varphi^3),  d\varphi^3=-\varphi^1\varphi^3,d\bar\varphi^3-\bar\varphi^1\bar\varphi^2\right)
	\]
	Now define $A$ to be the sub-cbba generated by $\varphi^1\bar\varphi^1$, $\varphi^2\bar\varphi^2$, $\varphi^3\bar\varphi^3$, i.e. as a bigraded vector space:
	\begin{align*}
	    A^{0,0}&=\C,\\
	    A^{1,1}&=\spn ( \varphi^1\bar\varphi^1,\,\varphi^2\bar\varphi^2,\,\varphi^3\bar\varphi^3)\\
	    A^{2,2}&=\spn( \varphi^1\varphi^2\bar\varphi^1\bar\varphi^2,\,\varphi^2\varphi^3\bar\varphi^2\bar\varphi^3,\,\varphi^1\varphi^3\bar\varphi^1\bar\varphi^3)\\
	    A^{2,1}&=\spn( \varphi^1\varphi^2\bar\varphi^3)\quad A^{1,2}=\spn(\varphi^3\bar\varphi^1\bar\varphi^2)\\
	    A^{3,3}&=\spn(\varphi^1\varphi^2\varphi^3\bar\varphi^1\bar\varphi^2\bar\varphi^3)
	\end{align*}
	with all other $A^{p,q}=0$ and only nontrivial differential $\del\delbar\varphi^3\bar\varphi^3=\varphi^1\varphi^2\bar\varphi^1\bar\varphi^2$. Then $A$ satisfies the $\del\delbar$-lemma, but $\langle \varphi^1\bar\varphi^1,\varphi^1\bar\varphi^1,\varphi^2\bar\varphi^2\rangle\neq 0$.
	\end{ex}
	
	We collect the relations between the various notions of formality in Figure 4. \\
	
		\begin{figure}$$
		{\tiny
			\begin{tikzcd}[row sep = large]
				& \textrm{ABC-geometric formality} \arrow[ld] \arrow[rd, "{\small \textrm{KT}}" description, dashed]                                                                            &                                       \\
				\textrm{weak formality} \arrow[rd, "{{\small \textrm{Hopf, KT}}}"', dashed, shift right] \arrow[dd] \arrow[rr, "{\small \textrm{KT}}" description, dashed, bend left] & \textrm{strong formality} \arrow[d] \arrow[l] \arrow[r]                                                                                                                       & \textrm{de Rham formality} \arrow[dd] \\
				& \deldelbar \textrm{--lemma} \arrow[ld, "{\small \textrm{ST}}" description, dashed] \arrow[d] \arrow[rd] \arrow[ru] \arrow[lu, "{\small \textrm{ST}}"', dashed, shift right=3] &                                       \\
				\parbox{3cm}{vanishing ABC--Massey products crossing the $\del\delbar$-bottleneck} \arrow[rr, "{\small \textrm{Stein}}" description, dashed, bend right, shift right = 1.5] \arrow[r, "{\small \textrm{KT}}", dashed]                                                                                                                                & \textrm{vanishing Dolbeault--Massey products} \arrow[r, "{\small \textrm{Stein}}" description, dashed, bend right, shift left = 2] \arrow[l, "{\small \textrm{ST}}" description, dashed, bend left, shift right = 2]                                                                                                                                & \textrm{vanishing Massey products} \arrow[ll, "{\small \textrm{ST}}" description, dashed, bend left, shift left = 5.5]
		\end{tikzcd} }$$
		\caption{Relations between various discussed notions. A dotted arrow means that the source property does not imply the target property, with a counterexample labelling the arrow. KT denotes the Kodaira--Thurston surface, ST the manifold constructed by Sferruzza--Tomassini, Hopf denotes the Hopf surface, and Stein denotes any Stein manifold with a non-vanishing Massey product \cite[p.192]{NT78}.}
	\end{figure}
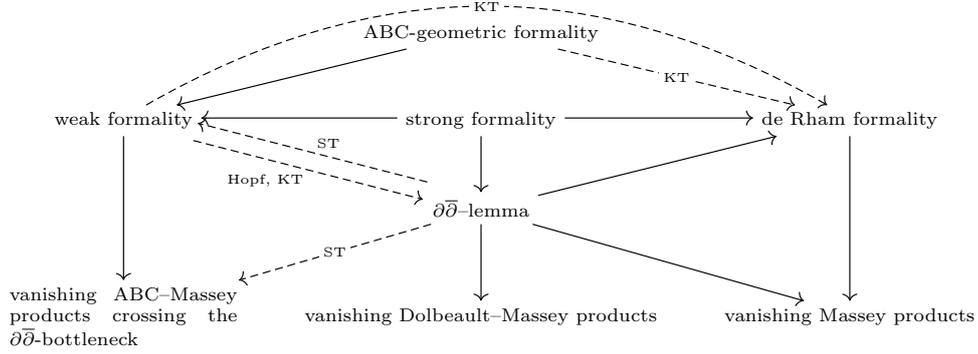
	
	The existence of the Sferruzza--Tomassini manifold leads us to the following question:
	\begin{question} Are compact K\"ahler manifolds strongly formal? \end{question}
	
	Either compact K\"ahler manifolds are strongly formal, in which case ABC--Massey products can tell the difference between K\"ahler manifolds and non--K\"ahler manifolds satisfying the $\deldelbar$-lemma, or we have potential invariants for K\"ahler manifolds.
	
	Since all K\"ahler nilmanifolds are biholomorphic to tori \cite[Theorem 2]{BC06}, one could consider looking at K\"ahler solvmanifolds for a negative answer to the above question. A complicating factor is that all such manifolds are finitely covered by tori \cite[Corollary 1]{Ar04}, and hence their non-formality could not be detected by a non-vanishing ABC--Massey triple product, see \Cref{masseyproductsandblowupssection}. 
	
	On the positive side, we have the following:
	
	\begin{prop}
	Hermitian symmetric spaces are strongly formal (even ABC-geometrically formal).
	\end{prop}
	
	\begin{proof} On Hermitian symmetric spaces, the cohomology can be computed from the sub-cbba of invariant forms, but it is known that $d\equiv 0$ on such forms. (This is the same proof as for usual (geometric) formality of Riemannian symmetric spaces, see e.g. \cite{Ko01}.)
	\end{proof}

There is also the following result in this direction:

\begin{thm}\cite[Theorem 3.5]{Ste23}
A compact K\"ahler manifold of complex dimension $n \geq 2$ with the Hodge diamond of a complete intersection is strongly formal.
\end{thm}

\subsection*{Formality and products}

We now consider how our bigraded notions of formality interact with products of complex manifolds (and more generally tensor products of cbba's). 

\begin{prop}
The product of strongly formal manifolds is strongly formal. The product of a strongly formal manifold with a weakly formal one is weakly formal.
\end{prop}

\begin{proof}
Given weak equivalences $\varphi:A\to C$ and $\psi: B\to D$, $\varphi\otimes\psi$ is a weak equivalence \cite{Ste23}. Iterating this, if $A$ and $B$ are connected by chains of quasi-isomorphisms to cbba's $H,K$, then also $A\otimes B$ is connected to $H\otimes K$. Now, if $\del_H\equiv 0\equiv\delbar_H$ then for any elements $h\in H$, $k\in K$, we have $\del(h\otimes k)=h\otimes\del_Kk$ and $\del(h\otimes k)=h\otimes\del_Kk$, hence the result follows.
\end{proof}

We show by example that the product of weakly formal manifolds need not be weakly formal:
\begin{ex} The product $KT \times KT$ of the Kodaira--Thurston surface with itself carries a non-trivial ABC--Massey triple product, though $KT$ is weakly formal by \Cref{KT}.

A bigraded model for $KT \times KT$ is given by $$\left( \Lambda(x, \bar{x}, y, \bar{y}, u, \bar{u}, v, \bar{v}), dy = x\bar{x}, dv = u\bar{u} \right),$$ where $x,y,u,v$ are in bidegree $(1,0)$. It is then a routine check that the ABC--Massey triple product $\langle x\bar{u}, u\bar{x}, \bar{x} \rangle$, represented by the Aeppli class $[y\bar{v}\bar{x}]$, is non-trivial. \end{ex}

\section{Massey products under blow-ups}\label{masseyproductsandblowupssection}

In this section we touch upon the interaction between blow-ups (and more generally non-zero degree holomorphic maps), and our notions of bigraded formality and Massey products. The following theorem allows one to construct new non-formal manifolds from existing ones:
	
	\begin{thm}
		Let $Z\subseteq X$ a complex submanifold of complex codimension $k\geq 2$ and $\pi:\tilde{X}\to X$ the blow-up of $X$ along $Z$. If $m$ is a non-trivial ABC--Massey product (either in the ad hoc or the spectral sequence sense) on $X$, living in total degree $<2k$, then also $\pi^*m$ is non-trivial.
	\end{thm}
	
	\begin{proof}
		Recall \cite{Ste21} that additively, there is the following formula for the cohomology of $\tilde{X}$:
		\begin{equation}\label{eqn: blow-up formula}
			H_{BC}(X)\oplus \bigoplus_{i=1}^{k-1}H_{BC}(Z)[i]\cong H_{BC}(\tilde{X}).
		\end{equation}
		Here $[i]$ denotes an up-right shift in bidegree by $(i,i)$. The induced ring structure up to degree $2k$ on the left hand side is simple to describe: Both summands have a natural $H_{BC}(Z)$-module structure (via the identity, resp. restriction) and elements in $H_{BC}(Z)$ multiply according to the multiplication in $H_{BC}(Z)$ with appropriate degree shift. In particular, if we truncate before degree $2k$, the right summand (which corresponds to $\ker\pi_*$ under the isomorphism), is an ideal. The analogous formulae hold for $H_A(\tilde{X})$ with its structure as an $H_{BC}(\tilde{X})$-module.

		The main idea is to construct a partial (bigraded) relative model for the blow-up, which retracts to $A_X$. More precisely, we will construct a map $\varphi:\cM\to A_{\tilde{X}}$ with the following properties:
		\begin{enumerate}
			\item As algebras, $\cM=A_X\otimes \Lambda V$ for some bigraded vector space $V$.
			\item The ideal $I(V)$ generated by $V$ is a differential ideal (i.e. $\cM=A_X\oplus I(V)$ as complexes).
			\item On $A_X$, $\varphi=\pi^*$.
			\item The map 
			\[
			H(\varphi):H(\cM)\to H(\tilde X)
			\]
			is an isomorphism in total degrees $<2k$, for $H=H_{BC}$ and $H=H_A$.
		\end{enumerate}
		
		With this setup, we obtain a commutative diagram of cbba's:
		\[
		\begin{tikzcd}
			A_X\ar[r]\ar[rd,"\pi^*"]&\cM\ar[r]\ar[d]&A_X\\
			&A_{\tilde{X}}&
		\end{tikzcd}
		\]
		in which the composition of the two maps (inclusion and projection) in the first row is the identity and the vertical map induces an isomorphism in cohomology up to degree $k$. This implies the claim by functoriality of Massey products. Note that if we had chosen basepoints of $X$ and $\tilde{X}$, the above is compatible with the induced augmentations.\\
		
		Let us indicate how to construct such a model. Since this is essentially a special case of the general construction of (bigraded, relative) models that is carried out in \cite{Ste23}, we allow ourselves to be brief. First, pick a bigraded basis $\{\mathfrak{b}_i\}$ for $H_{BC}(Z)$ and a collection of linearly independent classes $\{\mathfrak{a}_j\}$ that span a complement to the image of the natural map $H_{BC}(Z)\to H_A(Z)$. Now pick representatives $b_i,a_j$ for the classes $j_*\pi^*\mathfrak{b}_i$ and $j_*\pi^*\mathfrak{a_i}$, where $j:E\to \tilde{X}$ is the inclusion of the exceptional divisor into $\tilde{X}$. We may choose the unique class in $H_{BC}^{0,0}(Z)$ to be $[1]$, which maps to $[\theta]$. Let us denote by $\theta$ a representative for that class. Then define $V'$ to be the vector space generated by the symbols (of pure bidegree) $x_i,y_j,\del y_j,\delbar y_j$ and set $\cM':=A_X\otimes \Lambda V'$ and define a differential to be as indicated by the symbols on the $y_j$ and zero on all other generators (i.e. $dx_i=d\del y_j=d\delbar y_j=0$). Then there is a well-defined map $\varphi':\cM'\to A_{\tilde{X}}$ by sending $x_i\mapsto b_i, y_i\mapsto a_j$. By \ref{eqn: blow-up formula}, the induced maps $H(\varphi')$ will be surjective and the pair $(\cM',\varphi')$ satisfies all conditions except maybe $4.$ Therefore, consider 
		\[
		C=\ker H_{BC}(\varphi'):H_{BC}^{<2k}(\cM')\to H_{BC}^{<2k}(\tilde{X})
		\]
		Under the identification $H_{BC}(\cM')=H_{BC}(X)\oplus H_{BC}(I(V'))$ and the formula $\ref{eqn: blow-up formula}$, we see that, since we are in degrees less than $2k$, $C=\{0\}\oplus C'$. (In fact, in degree $2k$, the class $\theta^k$ will create a relation between both summands). Therefore, in order to kill these relations, we may pick (generators $c_i$ for) classes in $\ker\pi_*$, and then define $V''$ to be the vector space given by the symbols $z_i,\del z_i,\delbar z_i$. Then set $\cM'':= \cM'\otimes\Lambda V''$, with differentials as indicated by the symbols and $\deldelbar z_i= c_i$. We may then extend $\varphi'$ to $\varphi'':\cM''\to A_{\tilde X}$ by sending $z_i$ to some $\deldelbar$-primitive for $\varphi'(c_i)$. The map $H(\cM')\to H(\cM'')$ has the property that all elements in $C$ map to $0$ and all elements in $\ker H_A(\varphi')$ map to elements in the image of $H_{BC}(\cM'')\to H_A(\cM'')$. Repeating this process if necessary, we obtain our desired pair $(V,\varphi)$ and hence $\cM$.
	\end{proof}
	
	\begin{ex}
	In \Cref{quadruple} we considered a complex fourfold with a non-trivial quadruple ABC--Massey product, in total degree 3. Applying the above theorem, we see that any blow-up of this fourfold will also carry a non-trivial quadruple ABC--Massey product, yielding non-nilmanifold examples of this phenomenon. \end{ex}
	
	\begin{rem}
		Essentially the same proof (with some obvious simplifications) shows the statement for ordinary Massey products in the de Rham cohomology of any blow-up of a submanifold with almost complex normal bundle (cf. \cite{M84}), in particular for symplectic blow-ups. This had previously been obtained by Babenko--Taimanov  \cite{BT00}.
	\end{rem}

	 In \cite{Ta10}, Taylor shows that a non-trivial triple Massey product pulls back non-trivially under a non-zero degree map of rational Poincar\'e duality spaces (in particular, closed orientable manifolds). An immediate adaptation of his argument to compact complex manifolds, where Aeppli and Bott--Chern cohomology are paired non-degenerately under Serre duality \cite{S07}, gives the following: 
	 
	 \begin{prop} Let $Y \xrightarrow{f} X$ be a non-zero degree holomorphic map of compact complex manifolds (e.g. any blow-down $\tilde{X} \to X$). A non-vanishing ABC--Massey triple product on $X$ pulls back to a non-vanishing ABC--Massey triple product on $Y$. \end{prop}
	 
	 \begin{proof} The argument is a direct translation of that in \cite{Ta10}; we indicate only which adaptations need to be made. In general, for Bott--Chern classes $x_1, x_3$ of bidegree $(p_1,q_1), (p_3, q_3)$ respectively, denote by $\mathcal{J}_{\{x_1\}, \{x_3\}}$ the subset of $H_A(X)$ given by $x_1H_A + H_Ax_3$, and denote by $\mathcal{A}_{\{x_1\}, \{x_3\}}$ the set of all $y \in H_{BC}$ such that $x_1y = x_3y = 0$. Then the ABC--Massey triple product determines a map $$\langle x_1, - , x_3 \rangle: \mathcal{A}_{\{x_1\}, \{x_3\}} \rightarrow H_A/\mathcal{J}_{\{x_1\}, \{x_3\}}.$$ 
	 
	 If $x_0, x_2$ are in $\mathcal{A}_{\{x_1\}, \{x_3\}}$, then $x_0 \langle x_1, x_2, x_3 \rangle$ is a single Aeppli cohomology class, since $x_0$ kills the indeterminacy, cf. \cite[Theorem 2.1]{Ta10}. In particular, if $x_0 \langle x_1, x_2, x_3 \rangle \neq 0$, then the ABC--Massey triple product $\langle x_1, x_2, x_3 \rangle$ is non-trivial.
	 
	 The non-degenerate pairing $H_A \otimes H_{BC} \to \CC$ induces a non-degenerate pairing $$\left( H_A/\mathcal{J}_{\{x_1\}, \{x_3\}} \right) \otimes \mathcal{A}_{\{x_1\}, \{x_3\}} \to \CC,$$ cf. \cite[Proposition 5.1]{Ta10}. 
	 
	 Now suppose $x_1, x_2, x_3$ are Bott--Chern classes on $X$ such that $\langle x_1, x_2, x_3 \rangle$ is non-trivial. Then $\langle x_1, x_2, x_3 \rangle$ is a non-trivial element in $H_A/\mathcal{J}_{\{x_1\}, \{x_3\}}$, and hence by the above there is a Bott--Chern class $t \in \mathcal{A}_{\{x_1\}, \{x_3\}}$ such that $t \langle x_1, x_2, x_3 \rangle$ (which is a single cohomology class) is the volume class of $X$, cf. \cite[Theorem 5.2]{Ta10}. Therefore the class $f^*t \in \mathcal{A}_{\{f^*x_1\}, \{f^*x_3\}}$ on $Y$ satisfies $\int_Y f^*t \langle f^*x_1, f^*x_2, f^*x_3 \rangle \neq 0$, implying that $\langle f^*x_1, f^*x_2, f^*x_3 \rangle$ is non-trivial on $Y$. \end{proof}

\subsection*{Conflict of interest statement.} On behalf of all authors, the corresponding author states that there is no conflict of interest.

\subsection*{Data availability statement.} Data sharing not applicable to this article as no datasets were generated or analysed during the current study.

\end{document}